\documentclass[preprint,12pt]{elsarticle}
\UseRawInputEncoding
\usepackage{amssymb}
\usepackage{amscd,amssymb,latexsym,srcltx}
\usepackage{dsfont}
\usepackage{amsthm}
\usepackage{setspace}
\usepackage{mdwlist}
\usepackage{enumerate}
\usepackage{graphics}
\usepackage{graphicx}
\usepackage{epsfig}
\usepackage{subfigure}
\usepackage{amssymb}
\usepackage{mathrsfs}
\usepackage{amsmath}
\allowdisplaybreaks
\usepackage{amsfonts}
\usepackage{geometry}
 \usepackage{hyperref}
\usepackage{geometry}
\biboptions{sort&compress}
\numberwithin{equation}{section}
\makeatother
\renewcommand{\-}{\hbox{-}}

\newtheorem{theorem}{Theorem}[section]
\newtheorem{assumption}[theorem]{Assumption}

\newtheorem{lemma}[theorem]{Lemma}

\theoremstyle{definition}
\newtheorem{definition}[theorem]{Definition}

\journal{}

\begin{document}

\begin{frontmatter}
\title{Asymptotic behavior of the wave equation with nonlocal weak damping, anti-damping and critical nonlinearity}

\author[label1]{Chunyan Zhao}
\ead{emmanuelz@163.com}
\author[label1]{Chengkui Zhong\corref{cor1}}
\ead{ckzhong@nju.edu.cn}
\author[label1]{Zhijun Tang}
\ead{tzj960629@163.com}
\address[label1]{Department of Mathematics, Nanjing University, Nanjing, 210093, China}

\cortext[cor1]{Corresponding author.}
\begin{abstract}
In this paper, we first prove an abstract theorem on the existence of polynomial attractors and the concrete estimate of their attractive velocity for infinite-dimensional dynamical systems, then apply this theorem to a class of wave equations with nonlocal weak damping and anti-damping in case that the nonlinear term~$f$~is of subcritical growth.
\end{abstract}


\end{frontmatter}

\section{Introduction}
In this paper, we investigate the existence of the global attractor for the wave equation with nonlocal weak damping, nonlocal weak anti-damping and critical nonlinear source term
\begin{align}
\label{wave equa}&u_{tt}-\Delta u+k||u_{t}||_{L^2(\Omega)}^p u_t+f(u)
=\displaystyle\int_{\Omega}K(x,y)u_{t}(y)dy+h(x) \ \ \text{in}\ [0,\infty)\times\Omega,\\
\label{boundary condition}&u=0 \ \text{on}\  [0,\infty)\times \partial\Omega,\\
\label{initial condition}&u(x,0)=u_0(x) , u_{t}(x,0)=u_{1}(x),\  x\in\Omega,
\end{align}
where $\Omega$ is a bounded domain in $\mathbb{R}^N(N\geq3)$ with smooth boundary $\partial\Omega$ and the following assumption holds:

\begin{assumption}\label{21-8-29-8}
\begin{itemize}
  \item [(i)]~$k$ and~$p$ are positive constants,~$K\in L^2(\Omega\times\Omega)$,~$h\in L^2(\Omega)$;
  \item [(ii)]~$f\in C^1(\mathbb{R})$~satisfies
  \begin{align}
\label{growth}&|f'(s)| \leq M(|s|^{\frac{2}{N-2}}+1),\\
\label{dissipativity condition}&\liminf_{|s|\rightarrow\infty}f'(s)\equiv \mu > -\lambda_{1},
\end{align}
where~$M\geq0$ and~$\lambda_{1}$ is the first eigenvalue of the operator~$-\Delta$ equipped with Dirichlet boundary condition.
\end{itemize}
\end{assumption}

Since the pioneering work of J.K. Hale et al. \cite{MR0466837}  on the dynamical behavior of dissipative wave equations in the 1970s, there has been a wealth of literature on the asymptotic state (with particular reference to  existence of the global attractor,  estimate of
its fractal dimension  and  existence of exponential attractors) of solutions of wave equations with various damping. Among them,  we refer  to \cite{Arrieta,Babin1992,MR2026182,MR2505696,MR1868930,Ghidaglia1987,MR2018135,MR908897,Ladyzhenskaya1991,Raugel,MR2455195,TEMAM}  for the wave equation with weak damping ~$ku_t$~which models
the oscillation process  occurring in many  physical systems, including electrodynamics, quantum mechanics, nonlinear elasticity, etc.
Wave equations with strong damping~$-k\Delta u_t$~(see \cite{patazelik} for their physical background) were studied in \cite{BP,MR1972247,Ghidaglia,Ghidaglia1987,patasquassina,patazelik}.  Literatures \cite{Feireisl,Feireisl2,MR1318582,MR1242226,MR1078761,MR2269940,MR1150828,MR2237675} were devoted to wave equations with nonlinear damping $g(u_t)$. The damping $(-\Delta)^{\alpha}u_t\ (\alpha\in(0,1))$ is called fractional damping. In particular,  it is referred to as structural damping when $1/2\leq\alpha<1$ and as moderate damping when $0\leq\alpha<1/2$. Studies related to wave equations with fractional damping can be found in~\cite{MR1972247,yangzhijiandingpengyan,yangzhijian} and references therein.

On the other hand, the long-time behavior of hyperbolic equations with nonlocal damping also has received great attention. For example, we refer to \cite{chueshov1013,yangzhijiandingpengyan1013} for the study of the Kirchhoff equation with the damping $M(||\nabla u||^2_{L^2(\Omega)})\Delta u_t$, to \cite{CavalcantiSilva,MR1608033} for the case of nonlocal weak damping $M(||\nabla u||_{L^2(\Omega)}^2)u_t$, to \cite{SilvaNarciso1013} for the case of the damping $M(||\nabla u||^2_{L^2(\Omega)})g(u_t)$, and to \cite{Chueshovstructural,SilvaNarciso10131} for the case of the damping  $M(||\nabla u||^2_{L^2(\Omega)})(-\triangle)^{\theta}u_t$. The damping terms involved in the references listed above all have Kirchhoff type coefficients~$M(||\nabla u||^2_{L^2(\Omega)})$.
In addition, Lazo \cite{MR2426592} proved the existence of a global solution to the equation
  \begin{equation*}
  u_{tt}+M(|A^{\frac{1}{2}}u|^2)Au+N(|A^{\alpha}u|^2)A^{\alpha}u_t=f,
\end{equation*}
where~$A$~is a positive self-adjoint operator defined in Hilbert space~$H$, $\alpha\in(0,1]$~and the functions~$M,N$~satisfy the nondegenerate condition.

While, to the best of our knowledge, only very few results are available for damped hyperbolic equations whose nonlocal damping coefficient depends on~$u_t$. Among them we highlight that in 1989 Balakrishnan and Taylor \cite{BalakrishnanTaylor} presented some extensible beam equations with nonlocal energy damping {\small$\left[\displaystyle\int_{\Omega}(|\Delta u|^{2}+|u_{t}|^{2})dx\right]^{q}\Delta u _{t}$} to model  the damping phenomena in flight structures.
 Recently Silva, Narciso and Vicente \cite{SilvaNarcisoVicente} have proved the global well-posedness, polynomial stability of the following beam model with the nonlocal energy damping
\begin{equation*}
u_{tt}-\kappa\Delta u+\Delta^{2}u-\gamma\left[\displaystyle\int_{\Omega}(|\Delta u|^{2}+|u_{t}|^{2})dx\right]^{q}\Delta u _{t}+f(u)=0.
\end{equation*}
Lazo \cite{Lazo} considered the local solvability of the wave equation
\begin{equation*}
  u_{tt}-M(\|\nabla u\|_{L^2(\Omega)}^2)\triangle u +N(\|u_t\|_{L^2(\Omega)}^2)u_t=b|u|^{p-1}u.
\end{equation*}

We are motivated by the literature mentioned above to study problem $(\ref{wave equa})$-$(\ref{initial condition})$ in our last work \cite{my1}. As far as we know, this constituted the first result  on the long-time behavior of wave equation with nonlocal damping~$ k||u_{t}||_{L^2(\Omega)}^p u_t$. In \cite{my1},  we have proved via the method of Condition (C) that  the system possesses a global attractor in the case that~$f$~satisfies the  subcritical growth condition. However, we did not solve the  problem of the existence of the global attractor for the critical case, which is the aim of the present paper.

Problem $(\ref{wave equa})$-$(\ref{initial condition})$ is a weakly damped model, in which the nonlocal coefficient~$k||u_{t}||^p$~reflects the effect of kinetic energy on damping in physics. The term $\int_{\Omega}K(x,y)u_{t}(y)dy$~is an anti-damping because it may provide energy. The difficulty of this problem lies first in  the nondegenerate, nonlocal coefficient of damping and the arbitrariness of the exponent $p>0$. Due to the influence of nonlocal coefficient~$k||u_{t}||^p$, when the velocity~$u_t$~is very small, the nonlocal damping is weaker than the linear damping. Furthermore, as  the velocity~$u_t$~is  smaller and~$p$ is larger, the damping is weaker and thus energy dissipation is  slower. In addition, the presence of the anti-damping term leads to the energy not decreasing along the orbit, and moreover, the effect of energy supplement brought by the anti-damping term needs to be overcome by the damping. All these factors cause difficulties in studying the long-term behavior of this model. At the same time, since~$f$~is of critical growth,  the corresponding Sobolev embedding is no longer compact, which makes all the methods based on compactness, including Condition~(C), no longer available to prove  the existence of the global attractor.

In this paper, to overcome the difficulty of lack of compactness in the critical case, we employ the criterion of asymptotic smoothness relying on the repeated inferior limit ~(see Lemma $\ref{lemma 1-9-1}$ below)~to prove the existence of the global attractor. Chueshov and Lasiecka \cite{Chueshov2008,Chueshov2010} proposed this criterion based on the idea of Khanmamdov \cite{Khanmamedov2006}. To handle the difficulty that nonlocal damping coefficient~$\|u_t\|^p$~brings in energy estimate, we use the strong monotone inequality for the general inner product space~(see Lemma $\ref{lemma l2.3}$ below). The proof of compactness  borrows  many ideas from \cite{Chueshov2008}.

As for the dissipativity,  A. Haraux \cite{Haraux}  obtained via barrier's method the uniform bound of the energy in terms of the initial energy for a dissipative wave equation with anti-damping. The key element of the method is that the dissipative term in the inequality for Lyapunov's function has a coefficient sublinearly dependent on the energy. I. Chueshov and I. Lasiecka \cite{Chueshov2008} further proved that systems whose Lyapunov's functions satisfy such inequalities are ultimately dissipative. Their strategy was to select the perturbation parameter $\epsilon$ in the energy inequality as a suitable function of the initial energy  according to the sublinear dependence of the coefficient of the dissipation term on the energy; and thus they deduced that the energy is ultimately bounded by a constant independent on the initial data. Following the method in \cite{Chueshov2008}, we prove the  dissipativity for problem $(\ref{wave equa})$-$(\ref{initial condition})$.

The establishment of the global well-posedness  follows the idea in \cite{Chueshov1,Chueshov2008}.

This paper is organized as follows. In Section $2$, we present
some notations and lemmas which will be needed later. In Section $3$ and Section $4$, we prove the global well-posedness and dissipativity of the dynamical system generated by problem $(\ref{wave equa})$-$(\ref{initial condition})$, respectively. In Section $5$, we establish the existence of the global attractor for this system.
\section{Preliminaries}
Throughout this paper, we will denote the inner product and the norm on $L^2(\Omega)$ by $(\cdot,\cdot)$ and $\|\cdot\|$, respectively, and the norm on $L^p(\Omega)$ by $\|\cdot\|_p$. The symbol $\mathcal{A}$ denotes the strictly positive operator on $L^2(\Omega)$ defined by $\mathcal{A}=-\triangle$ with domain $D(\mathcal{A})=H^2(\Omega)\cap H^1_0(\Omega)$.
The symbols $\hookrightarrow$ and $\hookrightarrow\hookrightarrow$ stand for continuous embedding and compact embedding, respectively.
The capital letter ``C" with a (possibly empty) set of subscripts will denote
a positive constant depending only on its subscripts and may vary from one
occurrence to another. And we write

\ \vspace*{-10pt} \[\Psi(u_t(t,x))=\displaystyle\int_{\Omega}K(x,y)u_t(t,y)dy.\]

In later sections, we will use the following Sobolev embeddings:
\[H^1_0(\Omega)\hookrightarrow L^{\frac{2N}{N-2}}(\Omega),~H^s(\Omega)\hookrightarrow L^{\frac{2N}{N-2s}(\Omega)}~(s\in(0.1)).\]

We then present some preliminaries.

\begin{definition}\cite{Chueshov2008}
Let~$\{S(t)\}_{t\geq0}$~be a semigroup on a complete metric space~$(X,d)$. A closed set~$\mathcal{B}\subseteq X$~is said to be absorbing for~$\{S(t)\}_{t\geq0}$~iff for any bounded
set~$B\subseteq X$~there exists~$t_0(B)$ such that~$S(t)B\subseteq \mathcal{B}$~for all~$t>t_0(B)$. The semigroup
 $\{S(t)\}_{t\geq0}$~is said to be dissipative iff it possesses a bounded absorbing set.
\end{definition}

\begin{definition} \cite{Chueshov2008}\label{def20-11-11-1}
A compact set~$\mathcal{A}\subseteq X$~is said to be a global attractor of the dynamical system~$(X, \{S(t)\}_{t\geq0})$~iff
\begin{enumerate}[(i)]
  \item~$\mathcal{A}\subseteq X$~is an invariant set, i.e.,~$S(t)\mathcal{A}=\mathcal{A}$~for all~$t\geq0$;
  \item~$\mathcal{A}\subseteq X$~is uniformly attracting, i.e., for all bounded set~$B\subseteq X$~we have
  \begin{equation*}
\lim_{t\rightarrow +\infty}\mathrm{dist}(S(t)B,\mathcal{A})=0.
\end{equation*}
Here and below~$\mathrm{dist}(A,B):=\sup_{x\in A}\mathrm{dist}_{X}(x,B)$~is the Hausdorff\break semi-distance.
\end{enumerate}
\end{definition}

\begin{definition} \cite{Chueshov2008}
A dynamical system~$(X, \{S(t)\}_{t\geq0})$
is said to be asymptotically smooth iff for any bounded set $B$ such that $S(t)B\subseteq B$ for $t>0$ there exists a compact set $K$ in the closure $\overline{B}$ of $B$, such that
  \begin{equation*}
\lim_{t\rightarrow +\infty}\mathrm{dist}(S(t)B,K)=0.
\end{equation*}
\end{definition}

\begin{lemma}\cite{Chueshov2008}\label{lemma 1-14-10}
Let $\big(X,\{S(t)\}_{t\geq0}\big)$ be a dissipative dynamical system, where the phase space $X$ is a complete metric space. Then $\big(X,\{S(t)\}_{t\geq0}\big)$~possesses a  global attractor if and only if $\big(X,\{S(t)\}_{t\geq0}\big)$~is asymptotically smooth.
\end{lemma}

\begin{lemma}\cite{Chueshov2010}\label{lemma 1-9-1}
Let $\big(X,\{S(t)\}_{t\geq0}\big)$ be a dynamical system, where the phase space $X$ is a complete metric space. Assume that for any bounded positively invariant set $B$ in $X$ and any $\epsilon>0$ there exists $T\equiv T(\epsilon,B)$ such that
\begin{equation}\label{22-2-18-41}
  \liminf_{m\rightarrow\infty} \liminf_{n\rightarrow\infty}\mathrm{dist}(S(T)y_n,S(T)y_m)\leq\epsilon\  \text{for every sequence}\ \{y_n\}\subseteq B.
  \end{equation}
  Then $\big(X,\{S(t)\}_{t\geq0}\big)$ is asymptotically smooth.
\end{lemma}

\begin{lemma}\cite{MR4064014}\label{lemma l2.3}
Let $\left(H,(\cdot,\cdot)_H\right)$ be an inner product space with the induced norm $\|\cdot\|_H$ and constant $p>1$. Then there exists some positive constant~$C_p$~such that for any $x,y\in H$ satisfying $(x,y)\neq (0,0)$, we have
\begin{eqnarray}\label{2.11}
&\big(\|x\|_H^{p-2}x-\|y\|_H^{p-2}y,x-y\big)_H
&\geq
 \begin{cases}
 C_p\|x-y\|_H^{p}, ~  p\geq2;\\
C_p\displaystyle\frac{\|x-y\|_H^{2}}{(\|x\|_H+\|y\|_H)^{2-p}},~1<p<2.
 \end{cases}
\end{eqnarray}
\end{lemma}

Inequality $(\ref{2.11})$, which was verified for $\mathbb{R}^N$ in \cite{Perai,Simon1978} and then for a general inner product space in \cite{MR4064014}, will play a crucial role in our estimate.

\begin{lemma}\cite{Simon1986}\label{lemma 1-11-1}
Assume $X\hookrightarrow\hookrightarrow B \hookrightarrow Y$ where $X,~B,~Y$ are Banach spaces. The following statements hold.
\begin{enumerate}[(i)]
\item  Let $F$ be bounded in $L^{p}(0,T;X)$ where $1\leq p<\infty$, and $\partial F/\partial t=\{\partial f/\partial t:f\in F\}$ be bounded in $L^{1}(0,T;Y)$, where $\partial /\partial t$ is the weak time derivative. Then $F$ is relatively compact in $L^{p}(0,T;B)$.
  \item Let $F$ be bounded in $L^{\infty}(0,T;X)$ and $\partial F/\partial t$ be bounded in $L^{r}(0,T;Y)$ where $r>1$. Then $F$ is relatively compact in $C(0,T;B)$.
\end{enumerate}
\end{lemma}

\section{Global well-posedness }
In this section we discuss the global well-posedness of problem $(\ref{wave equa})$-$(\ref{initial condition})$.
We will use the following definitions of solutions.
\begin{definition}
A function $u(t)\in C([0,T];H^1_0(\Omega)) \cap C^1([0,T];L^2(\Omega))$ with $u(0)=u_0$ and $u_t(0)=u_1$ is said to be\hfil\break
$(i)$ strong solution to problem $(\ref{wave equa})$-$(\ref{initial condition})$ on the interval $[0,T]$, iff\\
$\bullet$ $u\in W^{1,1}(a,b;H^1_0(\Omega))$ and $u_t\in W^{1,1}(a,b;L^2(\Omega))$ for any $0<a<b<T$;\\
$\bullet$ $-\Delta u(t)+k||u_{t}(t)||_{L^2(\Omega)}^p u_t(t)\in L^2(\Omega)$ for almost all $t\in[0,T]$;\hfil\break
$\bullet$ equation $(\ref{wave equa})$ is satisfied in $L^2(\Omega)$ for almost all $t\in[0,T]$;\hfil\break
$(ii)$ generalized solution to problem $(\ref{wave equa})$-$(\ref{initial condition})$ on the interval $[0,T]$, iff there exists a sequence of strong solutions $\{u^{j}(t)\}$ to problem $(\ref{wave equa})$-$(\ref{initial condition})$ with initial data $(u_0^{j},u_1^{j})$ instead of $(u_{0},u_{1})$ such that $(u^{j},u_t^{j})\rightarrow (u,u_t)$  in $ C([0,T];H^1_0(\Omega)\times L^2(\Omega))$ as $j\rightarrow +\infty$;\hfil\break
$(iii)$ weak solution  to problem $(\ref{wave equa})$-$(\ref{initial condition})$ on the interval $[0,T]$, iff
\begin{equation}
\begin{split}\label{6-12-1}
\int_{\Omega}u_{t}(t,x)\psi(x)dx =&\int_{\Omega}u_{1}\psi dx + \int^{t}_{0}\left[\int_{\Omega\times\Omega}K(x,y)u_{t}(\tau,y)\psi(x)dxdy\right.\\
&\  +\int_{\Omega}h(x)\psi(x)dx -\int_{\Omega}\nabla u(\tau,x) \nabla \psi(x) dx  \\
&\ \left.-k||u_{t}(\tau)||^p \int_{\Omega}u_t(\tau,x)\psi(x) dx- \int_{\Omega}f(u(\tau,x))\psi(x)dx\right]d\tau
\end{split}
\end{equation}
holds for every $\psi \in H^1_0(\Omega) $ and for almost all $t \in [0,T]$.
\end{definition}
In order to prove the global well-posedness for  $(\ref{wave equa})$-$(\ref{initial condition})$,  we need the following definitions and lemmas.
\begin{definition}\cite{R. Showalter}
Let $X$ be a real Banach space with dual space $X^*$, $F:X\rightarrow X^*$ is said to be monotone if
$\left<F(u)-F(v),u-v\right>\geq 0,\ \forall u,v\in X$, and hemicontinuous if for each $u,v\in X$ the real-valued function $\lambda\rightarrow \left<F(u+\lambda v),v\right>$ is continuous.
\end{definition}
\begin{definition}\cite{R. Showalter}
Let $X$ and $Y$ be Banach spaces. $F:X\rightarrow Y$ is called demicontinuous if it is continuous from
$X$ with norm convergence to $Y$ with weak convergence.
\end{definition}
\begin{lemma}\cite{K. Deimling}\label{lem22-2-18-3}
Let $X$ be a real reflexive Banach space, $F:X\rightarrow X^*$ hemicontinuous, monotone and coercive, i.e. $\frac{\left<Fx,x\right>}{||x||}\rightarrow\infty$ as $||x||\rightarrow\infty$. Then $F$ is onto $X^*$.
 \end{lemma}
 \begin{lemma}\cite{R. Showalter}\label{lem22-2-18-30}
Let $X$ be a reflexive Banach space. If $F:X\rightarrow X^*$ hemicontinuous, monotone and bounded, then it is demicontinuous.
\end{lemma}
\begin{lemma}\cite{Chueshov1} \label{lem22-2-18-1}
Let $A:D(A)\subseteq H\rightarrow H$ be a maximal accretive operator on a Hilbert space $H$, i.e., $\left(Ax_1-Ax_2,x_1-x_2\right)_H\geq 0$ for any $x_1,x_2\in D(A)$ and $\mathrm{Rg}(I+A)=H$; besides, assume that $0\in A0$. Let $B:H\rightarrow H$ be a locally Lipschitz. If $u_0\in D(A)$, $f\in W^{1,1}(0,t;H)$ for all $t>0$, then there exists $t_{\mathrm{max}}\leq +\infty$ such that the initial value problem
\begin{equation}\label{6-13-33}
  u_t+Au+Bu\ni f\ \  \text{and} \ \ u=u_0\in H
\end{equation}
has a unique strong solution $u$ on the interval $[0,t_{\mathrm{max}})$.

Whereas, if $u_0\in \overline{D(A)}$, $f\in L^{1}(0,t;H)$ for all $t>0$, then problem  $(\ref{6-13-33})$
has a unique generalized solution $u\in C([0,t_{\mathrm{max}});H)$.

 Moreover, in both cases we have $\lim\limits_{t\rightarrow t_{\mathrm{max}}}\|u(t)\|_H=\infty $ provided $t_{\mathrm{max}}<\infty$.
\end{lemma}
We are now ready to establish the global well-posedness for $(\ref{wave equa})$-$(\ref{initial condition})$.
\begin{theorem}\label{thm22-2-18-1}
Let $T>0$ be arbitrary. Under Assumption $\ref{21-8-29-8}$, we have the following statements.
\begin{enumerate}[(i)]
\item  For any $(u_0,u_1) \in H^1_0(\Omega)\times H^1_0(\Omega)$ such that $-\Delta u_0+k||u_1||^p u_1\in L^2(\Omega)$, there exists a unique strong solution
  $u(t)$ to problem $(\ref{wave equa})$-$(\ref{initial condition})$ on $[0,T]$.
\item For every $(u_0,u_1) \in H^1_0(\Omega)\times L^2(\Omega)$ there exists a unique generalized solution, which is also the weak solution to problem $(\ref{wave equa})$-$(\ref{initial condition})$.
\end{enumerate}
\end{theorem}

\begin{proof}
We divide our proof into three steps.

\vspace*{4pt}\noindent\textbf{Step $1$.} We first prove local well-posedness of problem $(\ref{wave equa})$-$(\ref{initial condition})$.

Let $U=(u,v)^T$ with $v=u_t$. We rewrite  $(\ref{wave equa})$-$(\ref{initial condition})$ as
\begin{equation}
\left\{
\begin{aligned}
 &U_{t}+A(U)=B(U),\ \ t>0,\\
 &U(0)=U_0,
\end{aligned}
  \right.
\end{equation}
where $U_0=(u_0,u_1)^T$, $A:D(A)\subseteq H^{1}_0(\Omega)\times L^{2}(\Omega)\rightarrow H^{1}_0(\Omega)\times L^{2}(\Omega)$ is given by
\begin{equation*}
\begin{split}
  &A(U)=\begin{pmatrix}-v\\-\Delta u+k||v||^p v\end{pmatrix},\\
  &U\in D(A)=\left\{(u,v)^T\in H^{1}_0(\Omega)\times H^{1}_0(\Omega)|-\Delta u+k||v||^p v\in L^2(\Omega)\right\}
\end{split}\end{equation*}
and $B:H^{1}_0(\Omega)\times L^{2}(\Omega)\rightarrow H^{1}_0(\Omega)\times L^{2}(\Omega)$ is given by
\begin{equation*}
  B(U)=\begin{pmatrix}0\\\Psi(v)+h(x)-f(u)\end{pmatrix},U=(u,v)^T\in H^{1}_0(\Omega)\times L^{2}(\Omega).
\end{equation*}
We note that
\begin{equation}\label{6-11-50}
\overline{D(A)}=H^{1}_0(\Omega)\times L^{2}(\Omega),
\end{equation}
because \[\left(H^{1}_0(\Omega)\bigcap H^{2}(\Omega)\right)\times \left(H^{1}_0(\Omega)\bigcap H^{2}(\Omega)\right)\subseteq D(A)\subseteq H^{1}_0(\Omega)\times L^{2}(\Omega)\] and $\left(H^{1}_0(\Omega)\bigcap H^{2}(\Omega)\right)\times \left(H^{1}_0(\Omega)\bigcap H^{2}(\Omega)\right)$ is dense in $H^{1}_0(\Omega)\times L^{2}(\Omega)$.

By Lemma~$\ref{lemma l2.3}$, for every $v_1, v_2$ in $L^2(\Omega)$ we have
\begin{equation}
\begin{aligned}\label{6-11-6}
\left(\|v_{1}\|^{p}v_{1}-\|v_{2}\|^{p}v_{2},v_{1}-v_{2}\right)
\geq0.
\end{aligned}
\end{equation}
Consequently, we have
\begin{equation}\label{6-10-30}
\begin{split}
  &\big(A(U_1)-A(U_2),U_1-U_2\big)_{H_{0}^{1}(\Omega)\times L^{2}(\Omega)}\\=&
  \big(\nabla(v_{2}-v_{1}),\nabla(u_{1}-u_{2})\big)+\big(\nabla(u_{1}-u_{2}),\nabla(v_{1}-v_{2})\big)\\
&+\big(k\|v_{1}\|^{p}v_{1}-k\|v_{2}\|^{p}v_{2},v_{1}-v_{2}\big)\\
\geq &0,
\end{split}
\end{equation}
for all $U_1,U_2\in D(A)$, where $U_1=(u_1,v_1)^T,U_2=(u_2,v_2)^T$.

We proceed to show that
\begin{equation}\label{6-11-19}
\mathrm{Rg}(I+A)=H^{1}_0(\Omega)\times L^{2}(\Omega),
\end{equation}
 i.e., for $\forall (f_0,f_1)^T\in H^{1}_0(\Omega)\times L^{2}(\Omega)$, the equation
\begin{equation}\label{6-10-33}
  (A+I)(U)=\begin{pmatrix}-v+u\\-\Delta u+k||v||^p v+v\end{pmatrix}=\begin{pmatrix}f_0\\f_1\end{pmatrix}
\end{equation}
has a solution.

Eliminating $u$ from $(\ref{6-10-33})$ gives
\begin{equation}\label{6-10-34}
  -\Delta v+k\|v\|^{p}v+v=f_{1}+\Delta f_{0}\in H^{-1}(\Omega).
\end{equation}
Define $G:H_{0}^{1}(\Omega)\rightarrow H^{-1}(\Omega)$ by
$G(v)=-\Delta v+k\|v\|^{p}v+v$ for each $v\in H_{0}^{1}(\Omega) $.
Obviously, for each $v_1,v_2\in H_{0}^{1}(\Omega)$,
\begin{equation*}
  \big<G(v_1+\lambda v_2),v_2\big>=\big(\nabla(v_1+\lambda v_2),\nabla v_2\big)+(1+k||v_1+\lambda v_2||^p)\big(v_1+\lambda v_2, v_2\big)
\end{equation*}
is a continuous function of real variable $\lambda$.

It follows from $(\ref{6-11-6})$ that
\begin{equation*}
\begin{split}
&\big<G(v_{1})-G(v_{2}),v_{1}-v_{2}\big> \\
   =&\|\nabla (v_{1}-v_{2})\|^{2} +\|v_{1}-v_{2}\|^{2}
   +k\big(\|v_{1}\|^{p}v_{1}-\|v_{2}\|^{p}v_{2},v_{1}-v_{2}\big)  \\
   \geq&0
  \end{split}
\end{equation*}
for all $v_1,v_2\in H_{0}^{1}(\Omega)$.

Moreover we have
\begin{equation}
\displaystyle\frac{\big<G(v),v\big>}{\|\nabla v\|}= \displaystyle\frac{\|\nabla v\|^{2}+
   \| v\|^{2}+k\|v\|^{p+2}}{\|\nabla v\|}
   \longrightarrow +\infty
\end{equation}
as $\|\nabla v\|\rightarrow\infty$.

In summary, $G$ is hemicontinuous, monotone and coercive. Thus, by Lemma $\ref{lem22-2-18-3}$, $G$ is onto $H^{-1}(\Omega)$  and $(\ref{6-11-19})$ follows immediately. Combining $(\ref{6-10-30})$ and $(\ref{6-11-19})$ means $A$ is m-accretive.

For any $u_1,u_2 \in H_{0}^{1}(\Omega)$, by $(\ref{growth})$ we have
\begin{equation}\label{local lip1}
\begin{split}
 &\|f(u_1)-f(u_2)\|\\=&\left\{\int_\Omega\bigg[\int_0^1f'\big(u_2+\theta(u_1-u_2)\big)\big(u_1-u_2\big)d\theta\bigg]^2dx\right\}^{\frac{1}{2}}\\
 \leq&C\left\{\int_\Omega\big(|u_1|^{\frac{4}{N-2}}+|u_2|^{\frac{4}{N-2}}+1\big)|u_1-u_2|^{2}dx\right\}^{\frac{1}{2}}\\
 \leq&C\big(\|u_1\|_{\frac{2N}{N-2}}^{\frac{2}{N-2}}+\|u_2\|_{\frac{2N}{N-2}}^{\frac{2}{N-2}}+1\big)\|u_1-u_2\|_{\frac{2N}{N-2}}\\
  \leq&C\big(\|\nabla u_1\|^{\frac{2}{N-2}}+\|\nabla u_2\|^{\frac{2}{N-2}}+1\big)\|\nabla(u_1-u_2)\|.
\end{split}\end{equation}

Moreover, using the H\"{o}lder inequality, we have
\begin{equation}\label{5-21-1}
\begin{split}
\left\|\int_{\Omega}K(x,y)\big(v_1(y)-v_2(y)\big)dy\right\| \leq\|K\|_{L^2(\Omega\times\Omega)}||v_1-v_2||\end{split}
 \end{equation}
for any $v_1,v_2\in L^{2}(\Omega)$.

It follows from $(\ref{local lip1})$ and $(\ref{5-21-1})$ that  $B$ is locally Lipschitz.

Now that we have proved that $A$ is m-accretive, $B$ is locally Lipschitz and $\overline{D(A)}=H^{1}_0(\Omega)\times L^{2}(\Omega)$, by Lemma $\ref{lem22-2-18-1}$, we conclude that there exists $t_{max}\leq+\infty$ such that problem $(\ref{wave equa})$-$(\ref{initial condition})$  has a unique strong solution $u$ on $[0,t_{max})$ for every $(u_0,u_1)\in D(A)$ and it has a unique generalized solution $u$ on $[0,t_{max})$ for every $(u_0,u_1)\in H_{0}^{1}(\Omega)\times L^{2}(\Omega)$, moreover $[0,t_{max})$ is the maximal interval on which the solution exists. Furthermore, for both the strong solution and the generalized solution  we have
\begin{equation}\label{6-11-99}
  \lim_{t\rightarrow t_{\mathrm{max}} }\|(u,u_t)\|_{H_{0}^{1}(\Omega)\times L^{2}(\Omega)}=\infty ,\  \text{provided}\  t_{\mathrm{max}}<+\infty.
\end{equation}

\noindent\textbf{Step $2$.} Next, we will prove the global well-posedness of problem $(\ref{wave equa})$-$(\ref{initial condition})$.

We denote
\begin{equation*}
\mathcal{E}(u(t),u_t(t))=\displaystyle\frac{1}{2}\Big(\|u_t\|^{2}+\|\nabla u\|^{2}\Big)+\displaystyle\int_{\Omega}F(u)dx-\displaystyle\int_{\Omega}hudx.
\end{equation*}

Choose $\mu_0\in\mathds{R}^+\cap(-\mu,\lambda_1)$. By $(\ref{dissipativity condition})$, there exists $M>0$ such that
\begin{equation*}
f'(s)>-\mu_0, \ |s|>M.
\end{equation*}
It follows that
\begin{equation*}
\begin{cases}
F(s)\geq-\displaystyle\frac{\lambda_1+\mu_0}{4}s^2-C, &|s|>M,\\
|F(s)|\leq C,&|s|\leq M.
\end{cases}
\end{equation*}
Consequently,
\begin{equation}\begin{split}\label{22-2-18-5}
\displaystyle\int_{\Omega}F(u)dx&\geq \displaystyle\int_{\Omega_1}\Big(-\displaystyle\frac{\lambda_1+\mu_0}{4}u^2-C\Big)dx+\displaystyle\int_{\Omega_2}F(u)dx\\
&\geq -\displaystyle\frac{\lambda_1+\mu_0}{4}\displaystyle\int_{\Omega}u^2dx-C_1,
\end{split}\end{equation}
where $\Omega_1=\big\{x\in\Omega:|u(x)|>M\big\}$, $\Omega_2=\big\{x\in\Omega:|u(x)|\leq M\big\}$ and $C_1$ is some positive constant.

It is easy to get
\begin{equation}\label{22-2-18-6}
\begin{split}
\left| \displaystyle\int_{\Omega}hudx\right|\leq \displaystyle\frac{1}{16}\left(1-\displaystyle\frac{\mu_0}{\lambda_1}\right)\|\nabla u\|^2+C.
\end{split}\end{equation}

By Poincar\'e's inequality we have
\begin{equation}\label{22-2-18-9}
\|\nabla u\|^{2}\geq \lambda_{1}\|u\|^2.
\end{equation}

By $(\ref{22-2-18-5})$-$(\ref{22-2-18-9})$ we have
\begin{equation}\label{6-11-70}
\begin{split}
&\mathcal{E}(u(t),u_t(t))\\ \geq&\displaystyle\frac{1}{2}\Big(\|u_t\|^{2}+\|\nabla u\|^{2}\Big)-\displaystyle\frac{\lambda_1+\mu_0}{4}\displaystyle\int_{\Omega}u^2dx-C_1-\displaystyle\frac{1}{8}\left(1-\displaystyle\frac{\mu_0}{\lambda_1}\right)\|\nabla u\|^2-C\\
\geq & \displaystyle\frac{1}{8}\left(1-\displaystyle\frac{\mu_0}{\lambda_1}\right)\Big(\|u_t\|^{2}+\|\nabla u\|^{2}\Big)-C.
\end{split}\end{equation}

We deduce from $(\ref{growth})$ that
\begin{equation}\label{22-2-18-23}
\begin{split}
\left| \displaystyle\int_{\Omega}F(u)dx\right|&\leq \displaystyle\int_{\Omega}C(1+|u|^{\frac{2N-2}{N-2}})dx\\
&\leq C\left(\|\nabla u\|^{\frac{2N-2}{N-2}}+1\right).
\end{split}\end{equation}
Using $(\ref{22-2-18-6})$ and $(\ref{22-2-18-23})$ we obtain
\begin{equation}\label{6-11-88}
\begin{split}
\mathcal{E}(u(t),u_t(t))\leq C\left(\|u_t\|^{2}+\|\nabla u\|^{2}+\|\nabla u\|^{\frac{2N-2}{N-2}}+1\right).
\end{split}\end{equation}
Multiplying (\ref{wave equa}) by $u_{t}$ and integrating on $\Omega$ yields
\begin{equation}\label{6-11-60}
 \displaystyle\frac{d}{dt}\mathcal{E}(u(t),u_t(t))=-k\|u_t\|^{p+2}+\int_{\Omega}\Psi(u_t)u_tdx
\end{equation}
for $t\in [0,t_{max})$.

Using Young's inequality with $\epsilon$, we deduce from $(\ref{6-11-60})$ that
 \begin{equation}
\begin{split}\label{6-11-66}
 \displaystyle\frac{d}{dt}\mathcal{E}(u(t),u_t(t))&\leq-k\|u_t\|^{p+2}+\|K\|_{L^{2}(\Omega\times\Omega)}\|u_t\|^{2}\\
 &\leq-k\|u_t\|^{p+2}+\displaystyle\frac{k}{2}\|u_t\|^{p+2}+C\\
 &\leq C
 \end{split}
\end{equation}
 for $t\in [0,t_{max})$.

Integrating $(\ref{6-11-66})$, we have
\begin{equation}\label{6-11-67}
  \mathcal{E}(u(t),u_t(t))\leq \mathcal{E}(u_0,u_1)+Ct.
\end{equation}
If $t_{\mathrm{max}}<+\infty$, we duduce from $(\ref{6-11-70})$,$(\ref{6-11-88})$ and $(\ref{6-11-67})$ that
{\small\begin{equation}\label{6-11-90}
  \| \left(u(t),u_t(t)\right)\|_{H^{1}_0(\Omega)\times L^{2}(\Omega)}^2\leq  C\left(\|u_1\|^{2}+\|\nabla u_0\|^{2}+\|\nabla u_0\|^{\frac{2N-2}{N-2}}+1+t_{\mathrm{max}}\right)<+\infty.
\end{equation}}By the definition of the generalized solution, it is easy to verify that $(\ref{6-11-90})$ is also true for the generalized solution.
Thus by $(\ref{6-11-99})$, we have proved the global existence and uniqueness of the strong solution as well as the generalized solution.

\vspace*{4pt}\noindent\textbf{Step $3$.} Finally, we will verify that the generalized solution to problem $(\ref{wave equa})$-$(\ref{initial condition})$ is also weak.

Obviously, the strong solution to problem $(\ref{wave equa})$-$(\ref{initial condition})$ is also weak.

Let $u(t)$ be the generalized solution to problem $(\ref{wave equa})$-$(\ref{initial condition})$, then by definition there exists a sequence of strong solutions $\{u^{j}(t)\}$ to problem $(\ref{wave equa})$-$(\ref{initial condition})$ with initial data $(u_0^{j},u_1^{j})$ instead of $(u_{0},u_{1})$ such that
\begin{equation}\label{6-12-15}
(u^{j},u_t^{j})\rightarrow (u,u_t)\  \text{in}\ \  C([0,T];H^1_0(\Omega)\times L^2(\Omega))\ \ \text{as}\ \  j\rightarrow+\infty.
\end{equation}
We have\begin{equation}
\begin{split}\label{22-2-18-37}
\int_{\Omega}u_{t}^j(t,x)\psi(x)dx =&\int_{\Omega}u_{1}^j\psi dx + \int^{t}_{0}\left[\int_{\Omega\times\Omega}K(x,y)u_{t}^j(\tau,y)\psi(x)dxdy\right.\\
&\  +\int_{\Omega}h(x)\psi(x)dx -\int_{\Omega}\nabla u^j(\tau,x) \nabla \psi(x) dx  \\
&\ \left.-k||u_{t}^j(\tau)||^p \int_{\Omega}u_t^j(\tau,x)\psi(x) dx- \int_{\Omega}f(u^j(\tau,x))\psi(x)dx\right]d\tau
\end{split}
\end{equation}
holds for every $\psi \in H^1_0(\Omega) $ and for almost all $t \in [0,T]$.

Define $D:L^{2}(\Omega)\rightarrow L^{2}(\Omega)$ by
$G(v)=\|v\|^{p}v$ for each $v\in L^{2}(\Omega)$. Inequality $(\ref{6-11-6})$ indicates that $D$ is accretive. Besides, it is apparent that $D$ is hemicontinuous and bounded. Consequently, due to Lemma $\ref{lem22-2-18-30}$, $D$ is demicontinuous. Thus, we have
\begin{equation}\label{6-12-10}
||u_{t}^j(\tau)||^p \int_{\Omega}u_t^j(\tau,x)\psi(x) dx\rightarrow ||u_{t}(\tau)||^p \int_{\Omega}u_t(\tau,x)\psi(x) dx\ \ \text{as}\ j\rightarrow +\infty.
\end{equation}
Since by $(\ref{6-12-15})$ there exists $J\in\mathds{N}$ such that $\max\limits_{\tau\in[0,T]}\|u_t^j(\tau)\|\leq \max\limits_{\tau\in[0,T]}\|u_t(\tau)\|+1$ for every $j\geq J$, we have
\begin{equation}\label{6-12-17}
\begin{split}
\left|||u_{t}^j(\tau)||^p \int_{\Omega}u_t^j(\tau,x)\psi(x) dx\right|\leq \left(\max\limits_{\tau\in[0,T]}\|u_t(\tau)\|+1\right)^{p+1}\|\psi\|
\leq C.
\end{split}
\end{equation}
By the Lebesgue convergence theorem, it follows from  $(\ref{6-12-10})$ and $(\ref{6-12-17})$ that
\begin{equation}\label{6-12-26}
\begin{split} \lim_{j\rightarrow +\infty}\int^{t}_{0}\left[||u_{t}^j(\tau)||^p \int_{\Omega}u_t^j(\tau,x)\psi(x) dx\right]d\tau=\int^{t}_{0}\left[||u_{t}(\tau)||^p \int_{\Omega}u_t^j(\tau,x)\psi(x) dx\right]d\tau.
\end{split}
\end{equation}
Letting $j\rightarrow +\infty$ in  $(\ref{22-2-18-37})$ and using  $(\ref{6-12-15})$ and  $(\ref{6-12-26})$, we see that $u(t)$ satisfies $(\ref{6-12-1})$, which completes the proof.
\end{proof}

By Theorem $\ref{thm22-2-18-1}$, problem $(\ref{wave equa})$-$(\ref{initial condition})$ generates an evolution semigroup $\{S(t)\}_{t\geq0}$ in the space $H^1_0(\Omega)\times L^2(\Omega)$ by the formula
$S(t)(u_0,u_1)=(u(t),u_t(t))$,  where $(u_0,u_1)\in H^1_0(\Omega)\times L^2(\Omega)$ and $u(t)$ is the weak solution to problem $(\ref{wave equa})$-$(\ref{initial condition})$.

\section{Dissipativity}
In this section, we will prove the dissipativity of the dynamical system generated by the weak solution to  problem $(\ref{wave equa})$-$(\ref{initial condition})$, which is a necessary condition for the existence of the global attractor.
 \begin{theorem}\label{dissipativity theom}
Under Assumption $\ref{21-8-29-8}$, the dynamical system $\big(H_{0}^{1}(\Omega)\times L^{2}(\Omega),\break\{S(t)\}_{t\geq0}\big)$
 generated by the weak solution of problem $(\ref{wave equa})$-$(\ref{initial condition})$ is dissipative, i.e., there exists $R>0$
 satisfying the property: for any bounded set $B$ in $H_{0}^{1}(\Omega)\times L^{2}(\Omega)$,
 there exists $t_{0}(B)$ such that $\|S(t)y\|_{H_{0}^{1}(\Omega)\times L^{2}(\Omega)}\leq R$
 for all $y\in B$ and $t\geq t_{0}(B)$.
\end{theorem}

\begin{proof}
Choose $\mu_0\in\mathds{R}^+\cap(-\mu,\lambda_1)$. By $(\ref{dissipativity condition})$, there exists $M>0$ such that
\begin{equation}\label{6-6-1}
f'(s)>-\mu_0, \ |s|>M.
\end{equation}
It follows that
\begin{equation*}
\begin{cases}
F(s)\geq-\displaystyle\frac{\lambda_1+\mu_0}{4}s^2-C, &|s|>M;\\
|F(s)|\leq C,&|s|\leq M.
\end{cases}
\end{equation*}
Consequently,
\begin{equation}\begin{split}\label{6-6-2}
\displaystyle\int_{\Omega}F(u)dx&\geq \displaystyle\int_{\Omega_1}\Big(-\displaystyle\frac{\lambda_1+\mu_0}{4}u^2-C\Big)dx+\displaystyle\int_{\Omega_2}F(u)dx\\
&\geq -\displaystyle\frac{\lambda_1+\mu_0}{4}\displaystyle\int_{\Omega}u^2dx-C_1,
\end{split}\end{equation}
where $\Omega_1=\big\{x\in\Omega:|u(x)|>M\big\}$, $\Omega_2=\big\{x\in\Omega:|u(x)|\leq M\big\}$ and $C_1$ is some positive constant.

Let
\begin{equation*}
V_{\epsilon}(t)=\displaystyle\frac{1}{2}\Big(\|u_t\|^{2}+\|\nabla u\|^{2}\Big)+\displaystyle\int_{\Omega}F(u)dx-\displaystyle\int_{\Omega}hudx+\epsilon \displaystyle\int_{\Omega}u_tudx.
\end{equation*}
Since by Poincar\'e's inequality we have
\begin{equation}\label{9-4-1}
\|\nabla u\|^{2}\geq \lambda_{1}\|u\|^2,
\end{equation}
there exists $\epsilon_0>0$ such that
\begin{equation}\label{6-6-9}
\left| \epsilon \displaystyle\int_{\Omega}u_tudx\right|\leq \displaystyle\frac{1}{16}\left(1-\displaystyle\frac{\mu_0}{\lambda_1}\right)\Big(\|u_t\|^2+\|\nabla u\|^2\Big)\end{equation}
holds for all $\epsilon\leq \epsilon_{0}$.

Hereafter we assume  $\epsilon \in (0,\epsilon_0)$.

We also have
\begin{equation}\label{6-7-1}
\left| \displaystyle\int_{\Omega}hudx\right|\leq \displaystyle\frac{1}{16}\left(1-\displaystyle\frac{\mu_0}{\lambda_1}\right)\|\nabla u\|^2+C.
\end{equation}
We deduce from $(\ref{growth})$ that
\begin{equation}\label{22-4-25-1}
\begin{split}
\left| \displaystyle\int_{\Omega}F(u)dx\right|\leq C\left(\|\nabla u\|^{\frac{2N-2}{N-2}}+1\right).
\end{split}\end{equation}

 We deduce from $(\ref{6-6-2})$-$(\ref{22-4-25-1})$ that
\begin{equation}\label{6-7-2}
\displaystyle\frac{1}{8}\left(1-\displaystyle\frac{\mu_0}{\lambda_1}\right)\Big(\|u_t\|^{2}+\|\nabla u\|^{2}\Big)-C\leq V_{\epsilon}(t) \leq C\Big(\|u_t\|^2+\|\nabla u\|^2+\|\nabla u\|^{\frac{2N-2}{N-2}}+1\Big).
\end{equation}

Multiplying (\ref{wave equa}) by $u_{t}+\epsilon u$ and integrating on $\Omega$ yields
\begin{equation}\begin{split}\label{6-6-3}
 \displaystyle\frac{d}{dt}V_{\epsilon}(t)=&-k\|u_t\|^{p+2}+\int_{\Omega}\Psi(u_t)u_tdx+\epsilon\bigg[-\|\nabla u\|^{2}+\|u_t\|^{2}\\&- k\|u_t\|^{p}\int_{\Omega}u_tudx-\displaystyle\int_{\Omega}f(u)udx+\int_{\Omega}\Psi(u_t)udx+\displaystyle\int_{\Omega}hudx\bigg].
\end{split}\end{equation}

We estimate the terms on the right hand side of identity $(\ref{6-6-3})$ as follows:
\begin{equation}\begin{split}\label{6-6-5}
\left|-k\|u_t\|^{p}\int_{\Omega}u_tudx\right|
 &\leq k C\|u_t\|^{p+1}\|\nabla u\|\\
 &\leq C k\Big(\|u_t\|^{p+1}\|\nabla u\|^{\frac{p}{p+2}}\Big)^
 {\frac{p+2}{p+1}}+\displaystyle\frac{1}{12}\left(1-\displaystyle\frac{\mu_0}{\lambda_1}\right)\|\nabla u\|^{\frac{2}{p+2}(p+2)}\\
 &=C k\|u_t\|^{p+2}\|\nabla u\|^{\frac{p}{p+1}}+\displaystyle\frac{1}{12}\left(1-\displaystyle\frac{\mu_0}{\lambda_1}\right)\|\nabla u\|^{2};
\end{split}\end{equation}
\begin{equation}\begin{split}\label{6-6-8}
\left|\int_{\Omega}\Psi(u_t)udx\right|
\leq \displaystyle\frac{1}{12}\left(1-\displaystyle\frac{\mu_0}{\lambda_1}\right)\|\nabla u\|^{2}+C\|u_t\|^{2};
\end{split}\end{equation}
\begin{equation}\label{6-6-7}
\left|\int_{\Omega}\Psi(u_t)u_tdx\right|\leq \|K\|_{L^{2}(\Omega\times\Omega)}\|u_t\|^{2}.
 \end{equation}
We infer from $(\ref{6-6-1})$ that
\begin{equation}\label{9-4-6}
F(s)\leq f(s)s+\displaystyle\frac{\mu_0}{2}s^2+C, \ \ |s|>M.
\end{equation}
By $(\ref{9-4-1})$ and $(\ref{9-4-6})$ we have
\begin{equation}\begin{split}\label{6-6-4}
-\displaystyle\int_{\Omega}f(u)udx
&\leq-\Big(\displaystyle\int_{\Omega}F(u)dx+\displaystyle\frac{\lambda_1+\mu_0}{4}\displaystyle\int_{\Omega}u^2dx+C_1\Big)\\
&\quad+\displaystyle\frac{1}{4}\left(\displaystyle\frac{3\mu_0}{\lambda_1}+1\right)\|\nabla u\|^{2}+C.
\end{split}\end{equation}

Using  $(\ref{6-6-2})$,  $(\ref{6-7-1})$, $(\ref{6-7-2})$,  $(\ref{6-6-5})$, $(\ref{6-6-8})$, $(\ref{6-6-7})$, $(\ref{6-6-4})$ and Young's inequality with $\epsilon$, we deduce from $(\ref{6-6-3})$ that
\begin{equation}\begin{split}\label{6-7-10}
&\frac{d}{dt}V_{\epsilon}(t)\\\leq&-k\|u_t\|^{p+2}\Big(1-C\epsilon\|\nabla u\|^{\frac{p}{p+1}}\Big)+\displaystyle\frac{k}{2}\| u_t\|^{p+2}+C\\&-\epsilon\bigg[\displaystyle\frac{1}{2}\left(1-\displaystyle\frac{\mu_0}{\lambda_1}\right)\Big(\|u_t\|^2+\|\nabla u\|^2\Big)+\Big(\displaystyle\int_{\Omega}F(u)dx+\displaystyle\frac{\lambda_1+\mu_0}{4}\displaystyle\int_{\Omega}u^2dx+C_1\Big)\bigg]\\
 \leq & -k\|u_t\|^{p+2}\bigg[\displaystyle\frac{1}{2}-C\epsilon\Big(V_{\epsilon}(t)+C\Big)^{\frac{p}{2(p+1)} } \bigg]-\displaystyle\frac{2}{3}\left(1-\displaystyle\frac{\mu_0}{\lambda_1}\right)\epsilon V_{\epsilon}(t)+C.
\end{split}\end{equation}

Integrating $(\ref{6-7-10})$ from $s$ to $t$ and rescaling  $\epsilon$, we have
\begin{equation}\label{6-8-41}
V_{\epsilon}(t)\leq e^{-\epsilon(t-s)} V_{\epsilon}(s)+\frac{C}{\epsilon}- \displaystyle\int_s^te^{-\epsilon(t-\tau)}k\|u_t\|^{p+2}\bigg[\displaystyle\frac{1}{2}-C\epsilon\Big(V_{\epsilon}(\tau)+C\Big)^{\frac{p}{2(p+1)} } \bigg]d\tau
\end{equation}
for all $t\geq s\geq 0$.

Inequality $(\ref{6-8-41})$ is exactly formula $(3.44)$ in Theorem 3.15 in \cite{Chueshov2008} with $b(\cdot)=C$ and $\gamma=\frac{p}{2(p+1)}$, and thus Theorem
3.15 in \cite{Chueshov2008} can be directly applied to gain the ultimate dissipativity of the dynamical system generated by the problem $(\ref{wave equa})$-$(\ref{initial condition})$.
\end{proof}

\section{The existence of the global attractor}
Having verified the dissipativity, by Lemma $\ref{lemma 1-14-10}$, in order to establish the  existence of the global attractor, we only need to prove that the system is asymptotically  smooth.  Further,  by Lemma $\ref{lemma 1-9-1}$, it is sufficient to verify inequality $(\ref{22-2-18-41})$. This is exactly what we do when proving the following theorem.

\begin{theorem}\label{thm20-12-26-1}
Under Assumption~$\ref{21-8-29-8}$, the dynamical system generated by the weak solution of problem $(\ref{wave equa})$-$(\ref{initial condition})$ possesses a global
attractor.
\end{theorem}
\begin{proof}
Let~$B$~be a positively invariant bounded set in $H_{0}^{1}(\Omega)\times L^{2}(\Omega)$.

For any sequence $\big\{(u_{0}^{(n)},u_{1}^{(n)})\big\}_{n=1}^{\infty}$ in $B$, we set $S(t) (u_{0}^{(n)},u_{1}^{(n)})=\big(u^{(n)}(t),\break u_t^{(n)}(t)\big)$. It follows from the positive invariance property of~$B$~that
\begin{equation}\label{4.3}
  \big\|\big(u^{(n)}(t),u_t^{(n)}(t)\big)\big\|_{ H^1_0(\Omega)\times L^2(\Omega)}\leq C_B,\ \forall t>0,n\in\mathbb{N}.
\end{equation}
Write
\begin{equation*}
E^{n,m}(t)=\frac{1}{2}\left[\|\nabla(u^{(n)}(t)-u^{(m)}(t))\|^2+\|u^{(n)}_{t}(t)-u^{(m)}_{t}(t)\|^2\right].
\end{equation*}

\vspace*{4pt}\noindent\textbf{Step 1.} We first estimate~$E^{n,m}(T)$.

The difference $u^{(n)}-u^{(m)}$ satisfies
\begin{equation}\label{1-9-2}
\begin{split}
&u^{(n)}_{tt}-u^{(m)}_{tt}-\triangle(u^{(n)}-u^{(m)})+k\|u^{(n)}_{t}\|^pu^{(n)}_{t}-k\|u^{(m)}_{t}\|^pu^{(m)}_{t}\\
  =&-f(u^{(n)})+f(u^{(m)})+\Psi(u^{(n)}_{t}-u^{(m)}_{t}).
\end{split}
\end{equation}
Multiplying $(\ref{1-9-2})$ by $(u^{(n)}_{t}(t)-u^{(m)}_{t}(t))$ in $L^2(\Omega)$ and then integrating from $t$ to $T$, we obtain
\begin{equation}\label{1-9-3}
\begin{split}
  &E^{n,m}(T)\\=&E^{n,m}(t)+\int_t^T\int_{\Omega}\Big[\big(\Psi(u^{(n)}_{t}(\tau)-u^{(m)}_{t}(\tau))\big)\big(u^{(n)}_{t}(\tau)-u^{(m)}_{t}(\tau)\big)\\
  &-\big(f(u^{(n)}(\tau))-f(u^{(m)}(\tau))\big)\big(u^{(n)}_{t}(\tau)-u^{(m)}_{t}(\tau)\big)\\&-\big(k\|u^{(n)}_{t}(\tau)\|^pu^{(n)}_{t}(\tau)-k\|u^{(m)}_{t}(\tau)\|^pu^{(m)}_{t}(\tau)\big)\big(u^{(n)}_{t}(\tau)-u^{(m)}_{t}(\tau)\big)\Big]dxd\tau.
\end{split}
\end{equation}
Integrating $(\ref{1-9-3})$ with respect to $t$ between $0$ and $T$  gives
\begin{equation}\label{1-9-4}
\begin{split}
  &T\cdot E^{n,m}(T)\\=&\int_0^TE^{n,m}(t)dt+\int_0^T\int_t^T\int_{\Omega}\Big[\big(\Psi(u^{(n)}_{t}(\tau)-u^{(m)}_{t}(\tau))\big)\big(u^{(n)}_{t}(\tau)-u^{(m)}_{t}(\tau)\big)
\end{split}
\end{equation}\begin{equation*}
\begin{split}  &-\big(f(u^{(n)}(\tau))-f(u^{(m)}(\tau))\big)\big(u^{(n)}_{t}(\tau)-u^{(m)}_{t}(\tau)\big)\\&-\big(k\|u^{(n)}_{\tau}(\tau)\|^pu^{(n)}_{t}(\tau)-k\|u^{(m)}_{t}(\tau)\|^pu^{(m)}_{t}(\tau)\big)\big(u^{(n)}_{t}(\tau)-u^{(m)}_{t}(\tau)\big)\Big]dxd\tau dt.
\end{split}
\end{equation*}
 Multiplying $(\ref{1-9-2})$ by $(u^{(n)}(t)-u^{(m)}(t))$ in $L^2(\Omega)$ and then integrating from $0$ to $T$, we obtain
\begin{equation}\label{1-9-5}
\begin{split}
  & \int_0^TE^{n,m}(t)dt\\=&-\frac{1}{2}\bigg[\int_{\Omega}\big(u^{(n)}_{t}(t)-u^{(m)}_{t}(t)\big)\big(u^{(n)}(t)-u^{(m)}(t)\big)dx\bigg]\bigg|^T_0\\
   &+\int_0^T\|u^{(n)}_t(t)-u^{(m)}_t(t)\|^2dt\\
  &+\frac{1}{2}\int_0^T\int_{\Omega}\Big[\big(\Psi(u^{(n)}_{t}(t)-u^{(m)}_{t}(t))\big)\big(u^{(n)}(t)-u^{(m)}(t)\big)\\
  &-\big(f(u^{(n)}(t))-f(u^{(m)}(t))\big)\big(u^{(n)}(t)-u^{(m)}(t)\big)\\
  &-\big(k\|u^{(n)}_{t}(t)\|^pu^{(n)}_{t}(t)-k\|u^{(m)}_{t}(t)\|^pu^{(m)}_{t}(t)\big)\big(u^{(n)}(t)-u^{(m)}(t)\big)\Big]dxdt .
\end{split}
\end{equation}
By Lemma~$\ref{lemma l2.3}$,
{\small\begin{equation}\label{1-10-2}
\int_0^T\int_t^T\int_{\Omega}\big(k\|u^{(n)}_{\tau}(\tau)\|^pu^{(n)}_{t}(\tau)-k\|u^{(m)}_{t}(\tau)\|^pu^{(m)}_{t}(\tau)\big)\big(u^{(n)}_{t}(\tau)-u^{(m)}_{t}(\tau)\big)dxd\tau dt\geq0.
\end{equation}}

Let $0<s<1$.
We infer from $(\ref{4.3})$ that
\begin{equation}\label{1-10-3}
  -\frac{1}{2}\bigg[\int_{\Omega}\big(u^{(n)}_{t}(t)-u^{(m)}_{t}(t)\big)\big(u^{(n)}(t)-u^{(m)}(t)\big)dx\bigg]\bigg|^T_0\leq C_B
\end{equation}
and
\begin{equation}\label{1-10-4}
\begin{split}
  &\int_0^T\int_{\Omega}-\big(k\|u^{(n)}_{t}(t)\|^pu^{(n)}_{t}(t)-k\|u^{(m)}_{t}(t)\|^pu^{(m)}_{t}(t)\big)\big(u^{(n)}(t)-u^{(m)}(t)\big)dxdt\\
  \leq&k\int_0^T\big(\|u^{(n)}_{t}(t)\|^{p+1}+\|u^{(m)}_{t}(t)\|^{p+1}\big)\cdot\|u^{(n)}(t)-u^{(m)}(t)\|dt\\
  \leq   &TC_B\sup_{t\in[0,T]}\|u^{(n)}(t)-u^{(m)}(t)\|\\
 \leq   &TC_B\sup_{t\in[0,T]}\|u^{(n)}(t)-u^{(m)}(t)\|_{H^s(\Omega)}.
\end{split}\end{equation}
By~$(\ref{growth})$~and~$(\ref{4.3})$, we have
\begin{equation}\label{1-10-5}
\begin{split}
 &\|f(u^{(n)}(t))-f(u^{(m)}(t))\|\\=&\left\{\int_\Omega\bigg[\int_0^1f'\big(u^{(m)}(t)+\theta(u^{(n)}(t)-u^{(m)}(t)\big)\big(u^{(n)}(t)-u^{(m)}(t)\big)d\theta\bigg]^2dx\right\}^{\frac{1}{2}}\\
 \leq&C\left\{\int_\Omega\big(|u^{(n)}(t)|^{\frac{4}{N-2}}+|u^{(m)}(t)|^{\frac{4}{N-2}}+1\big)|u^{(n)}(t)-u^{(m)}(t)|^{2}dx\right\}^{\frac{1}{2}}\\
 \leq&C\big(\|u^{(n)}(t)\|_{\frac{2N}{N-2}}^{\frac{2}{N-2}}+\|u^{(m)}(t)\|_{\frac{2N}{N-2}}^{\frac{2}{N-2}}+1\big)\|u^{(n)}(t)-u^{(m)}(t))\|_{\frac{2N}{N-2}}
\end{split}\end{equation} \begin{equation*}
\begin{split} \leq&C\big(\|\nabla u^{(n)}(t)\|^{\frac{2}{N-2}}+\|\nabla u^{(m)}(t)\|^{\frac{2}{N-2}}+1\big)\|\nabla(u^{(n)}(t)-u^{(m)}(t)\|\\
  \leq &C_B.
\end{split}\end{equation*}
Consequently,
\begin{equation}\label{1-10-6}
\begin{split}
&\int_0^T\int_{\Omega}-\big(f(u^{(n)}(t))-f(u^{(m)}(t))\big)\big(u^{(n)}(t)-u^{(m)}(t)\big)dxdt\\
\leq &\int_0^T\|f(u^{(n)}(t))-f(u^{(m)}(t))\|\cdot\|u^{(n)}(t)-u^{(m)}(t)\|dt\\
\leq &TC_B\sup_{t\in[0,T]}\|u^{(n)}(t)-u^{(m)}(t)\|\\
\leq &TC_B\sup_{t\in[0,T]}\|u^{(n)}(t)-u^{(m)}(t)\|_{H^s(\Omega)}.
\end{split}\end{equation}
By Lemma~$\ref{lemma l2.3}$, for any $\epsilon>0$, we have
\begin{equation}\label{1-10-7}
\begin{split}
&\|u^{(n)}_t(t)-u^{(m)}_t(t)\|^2\\
\leq &\frac{\epsilon}{2}+C_{\epsilon}\|u^{(n)}_t(t)-u^{(m)}_t(t)\|^{p+2}\\
\leq&\frac{\epsilon}{2}+C_{\epsilon}k\int_{\Omega}\big(\|u^{(n)}_{t}(t)\|^pu^{(n)}_{t}(t)-\|u^{(m)}_{t}(t)\|^pu^{(m)}_{t}(t)\big)\big(u^{(n)}_t(t)-u^{(m)}_t(t)\big)dx.
\end{split}
\end{equation}
We deduce from~$(\ref{4.3})$, $(\ref{1-9-3})$~and~$(\ref{1-10-7})$~that
\begin{equation}\label{1-10-8}
\begin{split}
&\int_0^T\|u^{(n)}_t(t)-u^{(m)}_t(t)\|^2dt\\
\leq&\frac{\epsilon}{2} T+C_{\epsilon}\bigg\{E^{n,m}(0)-E^{n,m}(T) \\ &+\int_0^T\int_{\Omega}\Big[\big(\Psi(u^{(n)}_{t}(t)-u^{(m)}_{t}(t))\big)\big(u^{(n)}_{t}(t)-u^{(m)}_{t}(t)\big)\\
  &-\big(f(u^{(n)}(t))-f(u^{(m)}(t))\big)\big(u^{(n)}_{t}(t)-u^{(m)}_{t}(t)\big)\Big]dxdt\bigg\}\\
  \leq&\frac{\epsilon}{2} T+C_{\epsilon,B}+C_{\epsilon}\int_0^T\int_{\Omega}\Big[\big(\Psi(u^{(n)}_{t}(t)-u^{(m)}_{t}(t))\big)\big(u^{(n)}_{t}(t)-\\& u^{(m)}_{t}(t)\big)
  -\big(f(u^{(n)}(t))-f(u^{(m)}(t))\big)\big(u^{(n)}_{t}(t)-u^{(m)}_{t}(t)\big)\Big]dxdt.
\end{split}\end{equation}

Plugging~$(\ref{1-9-5})$, $(\ref{1-10-2})$, $(\ref{1-10-3})$, $(\ref{1-10-4})$, $(\ref{1-10-6})$~and~$(\ref{1-10-8})$~into ~$(\ref{1-9-4})$, we obtain
\begin{equation}\label{1-10-9}
\begin{split}
  &E^{n,m}(T)\\\leq&\frac{C_{\epsilon,B}}{T}+\frac{\epsilon}{2}+C_B\sup_{t\in[0,T]}\|u^{(n)}(t)-u^{(m)}(t)\|_{H^s(\Omega)}\\&+C_{\epsilon,B}\frac{1+T}{T}\int_0^T\|\Psi(u^{(n)}_{t}(t)-u^{(m)}_{t}(t))\|dt\\& + \frac{C_{\epsilon}}{T}\Bigg[\left|\int_0^T\int_{\Omega}\big(f(u^{(n)}(t))-f(u^{(m)}(t))\big)\big(u^{(n)}_t(t)-u^{(m)}_t(t)\big)dxdt\right|
  \\&+\left|\int_0^T\int_t^T\int_{\Omega}
  \Big(f(u^{(n)}(\tau))-f(u^{(m)}(\tau))\big)\big(u^{(n)}_{t}(\tau)-u^{(m)}_{t}(\tau)\big)dxd\tau dt\right|\Bigg].
\end{split}
\end{equation}

\noindent
\textbf{Step 2.} Next, we will investigate some convergence properties of the terms on the right in $(\ref{1-10-9})$.

By Alaoglu's theorem and Lemma $\ref{lemma 1-11-1}$, we deduce from $(\ref{4.3})$ and $H^1_0(\Omega)\hookrightarrow\hookrightarrow H^s(\Omega)\hookrightarrow L^2(\Omega)$ that there exists a subsequence of~$\big\{(u^{(n)},u^{(n)}_t)\big\}_{n=1}^{\infty}$, still denoted by~$\big\{(u^{(n)},u^{(n)}_t)\big\}_{n=1}^{\infty}$, such that
\begin{eqnarray}\label{1-13-1}
 \begin{cases}
 (u^{(n)},u^{(n)}_t)\overset{\ast}\rightharpoonup (u,v)\ \  \text{in} \ L^{\infty}(0,T;H^1_0(\Omega)\times L^2(\Omega)),\\
 u^{(n)}\rightarrow w\ \   \text{in }\ C([0,T];H^s(\Omega)),
 \end{cases}\  \ \text{as}\ n\rightarrow\infty.
\end{eqnarray}
Moreover, we can verify that $v=u_t$ and $w=u$.
Indeed, by $(\ref{1-13-1})$, for any~$\phi(s)\in C_c^{\infty}[0,t]$~and any~$\psi_0(x)\in H^2(\Omega)\cap H^1_0(\Omega)$, we have
\begin{equation*}
\begin{split}
  &\int_0^t\big(u^{(n)}_t(s),\phi(s)\triangle\psi_0(x)\big)ds\\=&\int_0^t\phi(s)\frac{d}{dt}\big(u^{(n)}(s),\triangle\psi_0(x)\big)ds\\=&-\int_0^t\phi'(s)\big( u^{(n)}(s),\triangle\psi_0(x)\big)ds\\=&\int_0^t\big( \nabla u^{(n)}(s),\phi'(s)\nabla\psi_0(x)\big)ds\\
\longrightarrow &\int_0^t\big(\nabla u(s),\phi'(s)\nabla\psi_0(x)\big)ds\\ =&\int_0^t\big(u_t(s),\phi(s)\triangle\psi_0(x)\big)ds
  \end{split}
\end{equation*}
and
\begin{equation*}
  \int_0^t\big(u^{(n)}_t(s),\phi(s)\triangle\psi_0(x)\big)ds\longrightarrow \int_0^t\big(v(s),\phi(s)\triangle\psi_0(x)\big)ds
\end{equation*}
as~$n\rightarrow\infty$. It follows that~$v=u_t$.

Since
\begin{equation*}
\begin{split}
  \int_0^T\big(\nabla(u^{(n)}(t)-w),\nabla\varphi\big)dt&= \int_0^T\big(\mathcal{A}^{\frac{s}{2}}(u^{(n)}(t)-w),\mathcal{A}^{1-\frac{s}{2}}\varphi\big)dt\\
 & \leq \sup_{t\in[0,T]}\|u^{(n)}(t)-w\|_{H^{s}(\Omega)}\int_0^T\|\varphi\|_{H^{2-s}(\Omega)}dt
\end{split}
\end{equation*}
holds for any $\varphi\in L^1(0,T;H^{2-s}(\Omega))$, by~$(\ref{1-13-1})$, we have~$\int_0^T\big(\nabla(u^{(n)}(t)-w),\nabla\varphi\big)dt\break\rightarrow 0$~
as~$n\rightarrow\infty$, which together with $(\ref{1-13-1})$ gives $w=u$.

Let $V$ be the completion of $L^2(\Omega)$ with respect to the norm $\|\cdot\|_{V}$ given by
$\|\cdot\|_V=\|\Psi(\cdot)\|+\|\mathcal{A}^{-\frac{1}{2}}\cdot\|$ and $W$ be the completion of $L^2(\Omega)$ with respect to the norm $\|\cdot\|_{W}$ given by $\|\cdot\|_W=\|\mathcal{A}^{-\frac{1}{2}}\cdot\|$. Since~$K\in L^2(\Omega\times\Omega)$, we have
\begin{equation}\label{1-14-2}
 L^2(\Omega)\hookrightarrow\hookrightarrow V\hookrightarrow W.
\end{equation}
Replacing $u^{(m)}(t)$ by $0$ in $(\ref{1-10-5})$ gives $\|f(u^{(n)}(t))-f(0)\|\leq C_B$, i.e., $\|f(u^{(n)}(t))\|\leq C_B$. In addition, it is easy to get
\begin{equation*}
  \|\Psi(u^{(n)}_{t}(t))\|\leq \|K\|_{L^2(\Omega\times\Omega)}\|u^{(n)}_{t}(t)\|\leq C_B.
\end{equation*}
Therefore, from~$(\ref{wave equa})$~we get
\begin{equation*}
\begin{split}
 &\|\mathcal{A}^{-\frac{1}{2}}u^{(n)}_{tt}(t)\|\\
 \leq&\|\nabla u^{(n)}(t)\|+k\|u^{(n)}_{t}(t)\|^p\|\mathcal{A}^{-\frac{1}{2}}u^{(n)}_{t}(t)\|+\|\mathcal{A}^{-\frac{1}{2}}\big(\Psi(u^{(n)}_{t}(t))+h-f(u^{(n)}(t))\big)\|\\
 \leq& C_B.
\end{split}
\end{equation*}
Consequently,
\begin{equation}\label{1-14-1}
   \int_0^T\|\mathcal{A}^{-\frac{1}{2}}u^{(n)}_{tt}(t)\|dt\leq C_{B,T}.
\end{equation}
Besides, we have
\begin{equation}\label{1-14-3}
  \int_0^T\|u^{(n)}_{t}(t)\|dt\leq C_{B,T}.
\end{equation}
By Lemma $\ref{lemma 1-11-1}$, $(\ref{1-14-2})$-$(\ref{1-14-3})$ imply that $\big\{u^{(n)}_{t}(t)\big\}_{n=1}^{\infty}$ is relatively compact in $L^1(0,\break T;V)$. Thus there exists a subsequence of~$\big\{(u^{(n)},u^{(n)}_t)\big\}_{n=1}^{\infty}$~(still denoted by itself) such that
\begin{equation}\label{1-14-4}
 \lim_{n,m\rightarrow\infty}\int_0^T\|\Psi\big(u^{(n)}_{t}(t)-u^{(m)}_{t}(t)\big)\|dt=0.
\end{equation}
In addition, it follows from~$(\ref{1-13-1})$~that
\begin{equation}\label{1-13-15}
  \lim_{n,m\rightarrow\infty}\sup_{t\in[0,T]}\|u^{(n)}(t)-u^{(m)}(t)\|_{H^{s}(\Omega)}=0,
\end{equation}
which together with $(\ref{1-14-4})$ gives
\begin{equation}\label{1-18-1}
\begin{split}
 I_1\equiv&\liminf_{n\rightarrow\infty} \liminf_{m\rightarrow\infty} \big[C_B\sup_{t\in[0,T]}\|u^{(n)}(t)-u^{(m)}(t)\|_{H^s(\Omega)}\\
 &+C_{\epsilon,B}\frac{1+T}{T}\int_0^T\|\Psi(u^{(n)}_{t}(t)-u^{(m)}_{t}(t))\|dt\big]\\
 =&0.
\end{split}
\end{equation}

Let~$F(\mu)=\displaystyle\int_0^{\mu}f(\tau)d\tau$. By~$(\ref{growth})$~and~$(\ref{4.3})$,
\begin{equation}\label{1-13-10}
\begin{split}
 & \left|\int_{\Omega}F(u^{(n)}(t))dx-\int_{\Omega}F(u(t))dx\right|\\
  \leq&\int_{\Omega}\left|\int_0^1f\big(u(t)+\theta(u^{(n)}(t)-u(t))\big)\cdot(u^{(n)}(t)-u(t))d\theta\right|dx\\
\leq & C\int_{\Omega}(|u^{(n)}(t)|^{\frac{N}{N-2}}+|u(t)|^{\frac{N}{N-2}}+1)\cdot|u^{(n)}(t)-u(t)|dx\\
\leq &C\|u^{(n)}(t)-u(t)\|\cdot\big(1+\|u^{(n)}(t)\|_{\frac{2N}{N-2}}^{\frac{N}{N-2}}+\|u(t)\|_{\frac{2N}{N-2}}^{\frac{N}{N-2}}\big)\\
\leq &C\|u^{(n)}(t)-u(t)\|_{H^{s}(\Omega)}\big(1+\|\nabla u^{(n)}(t)\|^{\frac{N}{N-2}}+\|\nabla u(t)\|^{\frac{N}{N-2}}\big)\\
\leq &C_B\|u^{(n)}(t)-u(t)\|_{H^{s}(\Omega)}
\end{split}
\end{equation}
holds for all~$t\geq0$.

Combining $(\ref{1-13-1})$ and $(\ref{1-13-10})$ gives
\begin{equation}\label{1-13-11}
  \int_{\Omega}F(u^{(n)}(t))dx\rightrightarrows\int_{\Omega}F(u(t))dx\ \text{as}\ n\rightarrow\infty.
\end{equation}
It follows from~$H^{N}(\Omega)\hookrightarrow L^{\infty}(\Omega)$~that~$L^1(\Omega)\hookrightarrow (L^{\infty}(\Omega))^{*}\hookrightarrow H^{-N}(\Omega)$. Hence we deduce from $(\ref{growth})$ and $(\ref{4.3})$ that
\begin{equation}\label{1-13-4}
\begin{split}
  &\|\mathcal{A}^{-\frac{N}{2}}f(u^{(n)}(t))-\mathcal{A}^{-\frac{N}{2}}f(u(t))\|\\[2mm]
= &\|f(u^{(n)}(t))-f(u(t))\|_{H^{-N}(\Omega)}\\[2mm]
\leq &C\|f(u^{(n)}(t))-f(u(t))\|_{1}\\[2mm]
\leq &C\int_{\Omega}\left|\int_0^1f'\big(\theta u^{(n)}(t)+(1-\theta)u(t)\big)\big(u^{(n)}(t)-u(t)\big)d\theta\right|dx\\[2mm]
\leq &C\int_{\Omega}(|u^{(n)}(t)|^{\frac{2}{N-2}}+|u(t)|^{\frac{2}{N-2}}+1)\cdot|u^{(n)}(t)-u(t)|dx\\[2mm]
\leq &C\|u^{(n)}(t)-u(t)\|\cdot(1+\|u^{(n)}(t)\|_{\frac{4}{N-2}}^{\frac{2}{N-2}}+\|u(t)\|_{\frac{4}{N-2}}^{\frac{2}{N-2}})\\[2mm]
\leq &C\|u^{(n)}(t)-u(t)\|_{H^{s}(\Omega)}(1+\|\nabla u^{(n)}(t)\|^{\frac{2}{N-2}}+\|\nabla u(t)\|^{\frac{2}{N-2}})\\[2mm]
\leq &C_B\|u^{(n)}(t)-u(t)\|_{H^{s}(\Omega)}
\end{split}
\end{equation}
holds for all~$t\geq0$.

 Combining~$(\ref{1-13-1})$~and~$(\ref{1-13-4})$~gives
\begin{equation}\label{20-11-28-1}
  \sup_{t\in[0,T]}\|\mathcal{A}^{-\frac{N}{2}}\big(f(u^{(n)}(t))-f(u(t))\big)\|\longrightarrow0 \ \text{as}\ n\rightarrow\infty.
\end{equation}
For each fixed~$t\in[0,T]$~and each~$\varphi\in L^1\big(0,T;H^N(\Omega)\cap H^1_0(\Omega)\big)$, we have
\begin{equation*}
\begin{split}
  &\int_t^T\big(f(u^{(n)}(\tau))-f(u(\tau)),\varphi\big)d\tau\ \  \\
  =&\int_t^T\big(\mathcal{A}^{-\frac{N}{2}}(f(u^{(n)}(\tau))-f(u(\tau))),\mathcal{A}^{\frac{N}{2}}\varphi\big)d\tau\\
  \leq &\sup_{\tau\in[0,T]}\|\mathcal{A}^{-\frac{N}{2}}\big(f(u^{(n)}(\tau))-f(u(\tau))\big)\|\int_0^T\|\varphi\|_{H^N(\Omega)}d\tau,
\end{split}
  \end{equation*}
  which, together with~$(\ref{20-11-28-1})$, gives
  \begin{equation}\label{1-13-6}
    \int_t^T\big(f(u^{(n)}(\tau))-f(u(\tau)),\varphi\big)d\tau\longrightarrow 0\ \text{as}\ n\rightarrow\infty.
  \end{equation}
 Since~$L^1\big(t,T;H^N(\Omega)\cap H^1_0(\Omega)\big)$~is dense in ~$L^1\big(t,T;L^{2}(\Omega)\big)$, $(\ref{1-13-6})$~implies
   \begin{equation}\label{1-13-7}
f(u^{(n)})\overset{\ast}\rightharpoonup f(u)\ \text{in}\ L^{\infty}\big(t,T;L^2(\Omega)\big)\ \text{as}\ n\rightarrow\infty.
\end{equation}
By~$(\ref{1-13-1})$, we have
\begin{equation}\label{1-15-1}
(u^{(n)},u^{(n)}_t)\overset{\ast}\rightharpoonup (u,u_t)\ \ \text{in} \ L^{\infty}\big(t,T;H^1_0(\Omega)\times L^2(\Omega)\big) \ \text{as}\ n\rightarrow\infty.
\end{equation}
From~$(\ref{1-13-7})$~and~$(\ref{1-15-1})$, we obtain
\begin{equation}\label{1-13-8}
\begin{split}
  &\lim_{n\rightarrow\infty}\lim_{m\rightarrow\infty}\int_t^T\int_{\Omega}
  f(u^{(n)}(\tau))u^{(m)}_{t}(\tau)dxd\tau
\end{split}\end{equation}\begin{equation*}
\begin{split}   =&\lim_{n\rightarrow\infty}\int_t^T\int_{\Omega}
  f(u^{(n)}(\tau))u_{t}(\tau)dxd\tau\\
=&\int_t^T\int_{\Omega}
  f(u(\tau))u_{t}(\tau)dxd\tau\\
=&\int_{\Omega}F(u(T))dx-\int_{\Omega}F(u(t))dx
\end{split}\end{equation*}
and
\begin{equation}\label{1-13-9}
\begin{split}
  \lim_{n\rightarrow\infty}\lim_{m\rightarrow\infty}\int_t^T\int_{\Omega}
  f(u^{(m)}(\tau))u^{(n)}_{t}(\tau)dxd\tau =\int_{\Omega}F(u(T))dx-\int_{\Omega}F(u(t))dx.
\end{split}\end{equation}
We deduce from~$(\ref{1-13-11})$, $(\ref{1-13-8})$~and ~$(\ref{1-13-9})$~that
\begin{equation}\label{1-13-12}
\begin{split}
  &\lim_{n\rightarrow\infty}\lim_{m\rightarrow\infty}\int_t^T\int_{\Omega}
  \big(f(u^{(n)}(\tau))-f(u^{(m)}(\tau))\big)\cdot\big(u^{(n)}_{t}(\tau)-u^{(m)}_{t}(\tau)\big)dxd\tau\\=&\lim_{n\rightarrow\infty}\lim_{m\rightarrow\infty}\Big[\int_{\Omega}F(u^{(n)}(T))dx-\int_{\Omega}F(u^{(n)}(t))dx+\int_{\Omega}F(u^{(m)}(T))dx\\& \qquad\qquad\quad-\int_{\Omega}F(u^{(m)}(t))dx-\int_t^T\int_{\Omega}
  f(u^{(m)}(\tau))u^{(n)}_{t}(\tau)dxd\tau\\&\qquad\qquad\quad-\int_t^T\int_{\Omega}
  f(u^{(n)}(\tau))u^{(m)}_{t}(\tau)dxd\tau \Big]\\=&0
\end{split}\end{equation}
for all~$t\in [0,T]$.

Due to~$(\ref{4.3})$~and~$(\ref{1-10-5})$,
\begin{equation}\label{1-13-13}
  \left|\int_t^T\int_{\Omega}
  \Big(f(u^{(n)}(\tau))-f(u^{(m)}(\tau))\Big)\cdot\Big(u^{(n)}_{t}(\tau)-u^{(m)}_{t}(\tau)\Big)dxd\tau\right|\leq C_{B,T}.
\end{equation}
By Lebesgue's dominated convergence theorem, combining~$(\ref{1-13-12})$~and~$(\ref{1-13-13})$~yields
\begin{equation}\label{1-13-14}
  \lim_{n\rightarrow\infty}\lim_{m\rightarrow\infty}\int_0^T\int_t^T\int_{\Omega}
  \Big(f(u^{(n)}(\tau))-f(u^{(m)}(\tau))\Big)\cdot\Big(u^{(n)}_{t}(\tau)-u^{(m)}_{t}(\tau)\Big)dxd\tau dt=0.
\end{equation}
It follows from~$(\ref{1-13-12})$~and~$(\ref{1-13-14})$~that
\begin{equation}\label{1-18-8}
\begin{split}
  I_2\equiv&\lim_{n\rightarrow\infty}\lim_{m\rightarrow\infty}\Bigg\{\frac{C_{\epsilon}}{T}\Bigg[\left|\int_0^T\int_{\Omega}\Big(f(u^{(n)}(t))-f(u^{(m)}(t))\Big)\Big(u^{(n)}_t(t)-u^{(m)}_t(t)\Big)dxdt\right|
 \\&+\left|\int_0^T\int_t^T\int_{\Omega}
  \Big(f(u^{(n)}(\tau))-f(u^{(m)}(\tau))\Big)\Big(u^{(n)}_{t}(\tau)-u^{(m)}_{t}(\tau)\Big)dxd\tau dt\right|\Bigg]\\
  &+\frac{C_{\epsilon,B}}{T}+\frac{\epsilon}{2}\Bigg\}\\=&\frac{C_{\epsilon,B}}{T}+\frac{\epsilon}{2}.
\end{split}
\end{equation}
We deduce from~$(\ref{1-10-9})$, ~$(\ref{1-18-1})$~and~$(\ref{1-18-8})$~that
\begin{equation*}
\begin{split}
\liminf_{m\rightarrow\infty} \liminf_{n\rightarrow\infty}E^{n,m}(T)\leq I_1+I_2=\frac{C_{\epsilon,B}}{T}+\frac{\epsilon}{2}\leq\epsilon
\end{split}\end{equation*}
for $T\geq\frac{2C_{\epsilon,B}}{\epsilon}$, which implies
\begin{equation*}
 \liminf_{m\rightarrow\infty} \liminf_{n\rightarrow\infty}\left\|\big(u^{(n)}(T),u^{(n)}_t(T)\big)-\big(u^{(m)}(T),u^{(m)}_t(T)\big)\right\|_{H^1_0(\Omega)\times L^2(\Omega)}\leq \sqrt{2\epsilon}.
 \end{equation*}
 Consequently, by Lemma~$\ref{lemma 1-9-1}$, the dynamical system generated by problem $(\ref{wave equa})$-$(\ref{initial condition})$ is asymptotically smooth. In addition,Theorem~$\ref{dissipativity theom}$~states that it is also dissipative. Thus by Lemma~$\ref{lemma 1-14-10}$~it possesses a global attractor.
\end{proof}

 \section*{Acknowledgment}
 The authors of this paper would like to express their
sincere thanks to the  referee for valuable comments.










\begin{thebibliography}{99}

\bibitem{Arrieta}
J. Arrieta, A. N. Carvalho and J. K. Hale, A damped hyperbolic equation with critical exponent, \emph{Comm. Partial Differential Equations}, \textbf{17} (1992), 841--866.

\bibitem{Babin1992}
A. V. Babin and M. I. Vishik, \emph{Attractors of Evolution Equations}, North-Holland Publishing Co., Amsterdam, 1992.

\bibitem{BalakrishnanTaylor}
A. V. Balakrishnan and L. W. Taylor, Distributed parameter nonlinear damping models for flight structures, \emph{Proceedings Damping 89}, Flight Dynamics Lab and Air Force Wright Aeronautical Labs, WPAFB, 1989.

\bibitem{MR2026182}
J. M. Ball, Global attractors for damped semilinear wave equations, \emph{Discrete Contin. Dyn. Syst.}, \textbf{10} (2004), 31--52.

\bibitem{BP}
V. Belleri and V. Pata, Attractors for semilinear strongly damped wave equations on {$\Bbb R^3$}, \emph{Discrete Contin. Dynam. Systems}, \textbf{7} (2001), 719--735.

\bibitem{MR1972247}
A. N. Carvalho and J. W. Cholewa, Attractors for strongly damped wave equations with critical nonlinearities, \emph{Pacific J. Math.}, \textbf{207} (2002), 287--310.

\bibitem{MR2505696}
A. N. Carvalho, J. W. Cholewa and T. Dlotko, Damped wave equations with fast growing dissipative nonlinearities, \emph{Discrete Contin. Dyn. Syst.}, \textbf{24} (2009), 1147--1165.

\bibitem{CavalcantiSilva}
M.  M. Cavalcanti, V. N. Domingos Cavalcanti, M. A. J.  Silva and C. M. Webler, Exponential stability for the wave equation with degenerate nonlocal weak damping, \emph{Israel J. Math.}, \textbf{219} (2017), 189--213.

\bibitem{MR1868930}
V. V. Chepyzhov and M. I. Vishik, \emph{Attractors for Equations of Mathematical Physics, Amer. Math. Soc.}, Providence, R.I., 2002.

\bibitem{Chueshovstructural}
I. Chueshov, Global attractors for a class of {K}irchhoff wave models with a structural nonlinear damping, \emph{J. Abstr. Differ. Equ. Appl.}, \textbf{1} (2010), 86--106.

\bibitem{chueshov1013}
I. Chueshov, Long-time dynamics of {K}irchhoff wave models with strong nonlinear damping, \emph{J. Differential Equations}, \textbf{252} (2012), 1229--1262.

\bibitem{Chueshov2015}
I. Chueshov, \emph{Dynamics of Quasi-stable Dissipative Systems}, Springer, Switzerland, 2015.

\bibitem{Chueshov1}
I. Chueshov, M. Eller and I. Lasiecka, On the attractor for a semilinear wave equation with critical exponent and nonlinear boundary dissipation,
\emph{Comm. Partial Differential Equations}, \textbf{27} (2002), 1901--1951.

\bibitem{Chueshov2008}
I. Chueshov and I. Lasiecka, \emph{Long-Time Behavior of Second Order Evolution Equations with Nonlinear Damping}, Mem. Amer. Math. Soc., 2008.

\bibitem{Chueshov2010}
I. Chueshov and I. Lasiecka, \emph{Von Karman Evolution Equations: Well-Posedness and Long Time Dynamics}, Springer Science \& Business Media, 2010.

\bibitem{SilvaNarciso10131}
M. A. J. da Silva and V. Narciso, Attractors and their properties for a class of nonlocal extensible beams, \emph{Discrete Contin. Dyn. Syst.}, \textbf{35} (2015), 985--1008.

\bibitem{SilvaNarciso1013}
M. A. J. da Silva and V. Narciso, Long-time dynamics for a class of extensible beams with nonlocal nonlinear damping, \emph{Evol. Equ. Control Theory}, \textbf{6} (2017), 437--470.

\bibitem{K. Deimling}
K. Deimling, \emph{Nonlinear Functional Analysis}, Springer, Berlin 1985.

\bibitem{yangzhijiandingpengyan1013}
P. Ding, Z. Yang and Y. Li, Global attractor of the {K}irchhoff wave models with strong nonlinear damping, \emph{Appl. Math. Lett.}, \textbf{76} (2018), 40--45.

\bibitem{Feireisl}
E. Feireisl, Attractors for wave equations with nonlinear dissipation and critical exponent, \emph{C. R. Acad. Sci. Paris S\'{e}r. I Math.}, \textbf{315} (1992), 551--555.

\bibitem{Feireisl2}
E. Feireisl, Finite-dimensional asymptotic behavior of some semilinear damped hyperbolic problems, \emph{J. Dynam. Differential Equations}, \textbf{6} (1994), 23--35.

\bibitem{MR1318582}
E. Feireisl, Global attractors for semilinear damped wave equations with supercritical exponent, \emph{J. Differential Equations}, \textbf{116} (1995), 431--447.

\bibitem{Ghidaglia}
J.M. Ghidaglia and A. Marzocchi,
Longtime behaviour of strongly damped wave equations, global attractors and their dimension, \emph{SIAM J. Math. Anal.}, \textbf{22} (1991), 879--895.

\bibitem{Ghidaglia1987}
J.M. Ghidaglia and R. Temam, Attractors for damped nonlinear hyperbolic equations, \emph{J. Math. Pures Appl.}, \textbf{66} (1987), 273--319.

\bibitem{MR2018135}
M. Grasselli and V. Pata, On the damped semilinear wave equation with critical exponent, \emph{Discrete Contin. Dyn. Syst.}, (2003), 351--358.

\bibitem{MR0466837}
J. K. Hale, \emph{Functional Differential Equations}, Springer-Verlag, New York, 1971.

\bibitem{MR908897}
J. K. Hale, Asymptotic behaviour and dynamics in infinite dimensions, in \emph{Nonlinear Differential Equations}, Pitman, Boston, MA, 1985, 1--42.

\bibitem{MR1242226}
J. K. Hale and G. Raugel, Attractors for dissipative evolutionary equations, \emph{International Conference on Differential Equations}, World Sci. Publ., River Edge, NJ, \textbf{1} (1993), 3--22.

\bibitem{Haraux}
A. Haraux, Two remarks on hyperbolic dissipative problems, \emph{Nonlinear Partial Differential Equations and their Applications}, \textbf{122} (1985), 161--179.

\bibitem{MR1078761}
A. Haraux, Semi-linear hyperbolic problems in bounded domains, \emph{Math. Rep.}, \textbf{3} (1987), 1--281.

\bibitem{SilvaNarcisoVicente}
M. A. Jorge Silva, V. Narciso and A. Vicente, On a beam model related to flight structures with nonlocal energy damping, \emph{Discrete Contin. Dyn. Syst. Ser. B}, \textbf{24} (2019), 3281--3298.

\bibitem{Khanmamedov2006}
A. K. Khanmamedov, Global attractors for von {K}arman equations with nonlinear interior dissipation, \emph{J. Math. Anal. Appl.}, \textbf{318} (2006), 92--101.

\bibitem{MR2269940}
A. K. Khanmamedov, Global attractors for wave equations with nonlinear interior damping and critical exponents, \emph{J. Differential Equations}, \textbf{230} (2006), 702--719.

\bibitem{Ladyzhenskaya1991}
O. Ladyzhenskaya, \emph{Attractors for Semigroups and Evolution Equations}, Cambridge University Press, 1991.

\bibitem{MR1608033}
H. Lange and G. Perla Menzala, Rates of decay of a nonlocal beam equation, \emph{Differential Integral Equations}, \textbf{10} (1997), 1075--1092.

\bibitem{Lax}
P. D. Lax, \emph{Functional Analysis}, Wiley-Interscience [John Wiley $\&$ Sons], New York, 2002.

\bibitem{Lazo}
P. P. D. Lazo, \emph{Quasi-linear Wave Equation with Damping and Source Terms}, Ph.D thesis, Federal University of Rio de Janeiro, Brazil, 1997.

\bibitem{MR2426592}
P. P. D. Lazo, Global solutions for a nonlinear wave equation, \emph{Appl. Math. Comput.}, \textbf{200} (2008), 596--601.

\bibitem{MR4079016}
Y. Li and Z. Yang, Optimal attractors of the {K}irchhoff wave model with structural nonlinear damping, \emph{J. Differential Equations}, \textbf{268} (2020), 7741--7773.

\bibitem{MWZH}
Q. Ma, S. Wang and C. Zhong, Necessary and sufficient conditions for the existence of global attractors for semigroups and applications, \emph{Indiana Univ. Math. J.}, \textbf{51} (2002), 1541--1559.

\bibitem{patasquassina}
V. Pata and M. Squassina, On the strongly damped wave equation, \emph{Comm. Math. Phys.}, \textbf{253} (2005), 511--533.

\bibitem{patazelik}
V. Pata and S. Zelik, Smooth attractors for strongly damped wave equations, \emph{Nonlinearity}, \textbf{19} (2006), 1495--1506.

\bibitem{Perai}
I. Perai, \emph{Multiplicity of Solutions for the $p$-Laplacian}, 1997.

\bibitem{MR1150828}
G. Raugel, Une \'{e}quation des ondes avec amortissement non lin\'{e}aire dans le cas critique en dimension trois, \emph{C. R. Acad. Sci. Paris S\'{e}r. I Math.}, \textbf{314} (1992), 177--182.

\bibitem{Raugel}
G. Raugel, Global attractors in partial differential equations, \emph{Handbook of Dynamical Systems}, Vol. 2, North-Holland, 2002, 885--982.

\bibitem{R. Showalter}
R. E. Showalter, \emph{Monotone Operators in Banach Spaces and Nonlinear Partial Differential Equations}, AMS, Providence, 1997.

\bibitem {Simon1978}
J. Simon, R\'egularit\'e de la solution d'une \'equation non lin\'eaire dans ${\bf R}\sp{N}$, \emph{Journ\'ees d'Analyse Non Lin\'eaire}, Lecture Notes in Math., Springer, Berlin, \textbf{665} (1978), 205--227.

\bibitem{Simon1986}
J. Simon, Compact sets in the space $L^p(0,T;B)$, \emph{Ann. Mat. Pura Appl.}, \textbf{146} (1987), 65--96.

\bibitem{MR2455195}
C. D. Sogge, \emph{Lectures on Non-linear Wave Equations}, 2$^{nd}$ edition, International Press, Boston, MA, 2008.

\bibitem{MR2237675}
C. Sun, M. Yang and C. Zhong, Global attractors for the wave equation with nonlinear damping, \emph{J. Differential Equations}, \textbf{227} (2006), 427--443.

\bibitem{TEMAM}
R. Temam, \emph{Infinite-Dimensional Dynamical Systems in Mechanics and Physics}, 2$^{nd}$ edition, Springer-Verlag, New York, 1997.

\bibitem{yangzhijiandingpengyan}
Z. Yang, P. Ding and L. Li, Longtime dynamics of the Kirchhoff equations with fractional damping and supercritical nonlinearity, \emph{J. Math. Anal. Appl.}, \textbf{442} (2016), 485--510.

\bibitem{yangzhijian}
Z. Yang, Z. Liu and P. Niu, Exponential attractor for the wave equation with structural damping and supercritical exponent, \emph{Commun. Contemp. Math.}, \textbf{18} (2016), 1550055, 13 pp.

\bibitem{MR4064014}
C. Zhao, C. Zhao and C. Zhong, The global attractor for a class of extensible beams with nonlocal weak damping, \emph{Discrete Contin. Dyn. Syst. Ser. B}, \textbf{25} (2020), 935--955.

\bibitem{my1}
C. Zhao, C. Zhao and C. Zhong, Asymptotic behaviour of the wave equation with nonlocal weak damping and anti-damping, \emph{J. Math. Anal. Appl.}, \textbf{490} (2020), 124186, 16 pp.

\end{thebibliography}
\end{document}